\documentclass[a4paper,twoside,reqno]{amsart}

\usepackage[pagewise]{lineno}
\usepackage[utf8]{inputenc}

\usepackage{authblk}
\usepackage{hyperref}
\hypersetup{urlcolor=blue, citecolor=red}
\usepackage{amsmath,amsthm,amssymb,xcolor,bbm,mathrsfs,graphicx, float}
\usepackage[english]{babel}
\usepackage{fancyhdr}
\newcommand{\norm}[1]{\|#1\|}

\makeatletter
\@namedef{subjclassname@1991}{2020 Mathematics Subject Classification}
\makeatother

\title{Garbage disposal game on finite graphs }
\author{Hsin-Lun Li}

\date{}
\email{hsinlunl@math.nsysu.edu.tw}

\theoremstyle{definition}
\newtheorem{theorem}{Theorem}

\newtheorem{lemma}[theorem]{Lemma}

\begin{document}

\allowdisplaybreaks

\thispagestyle{firstpage}
\maketitle
\begin{center}
    Hsin-Lun Li
    \centerline{$^1$National Sun Yat-sen University, Kaohsiung 804, Taiwan}
\end{center}
\medskip

\begin{abstract}
    The garbage disposal game involves a finite set of individuals, each of whom updates their garbage by either receiving from or dumping onto others. We examine the case where only social neighbors, whose garbage levels differ by a given threshold, can offload an equal proportion of their garbage onto others. Remarkably, in the absence of this threshold, the garbage amounts of all individuals converge to the initial average on any connected social graph that is not a star.
\end{abstract}

\section{Introduction}
The garbage disposal game comprises a set of \( n \) individuals. Each individual updates their garbage either by receiving garbage from others or by dumping garbage onto others~\cite{hirai2006coalition}. Mathematically, let \( [n] = \{1, \ldots, n\} \) represent the set of all individuals, and let \( x_i(t) \geq 0 \) denote the amount of garbage held by individual \( i \) at time \( t \). The update rule for individual \( i \)'s garbage is given by \( x_i(t+1) = \sum_{j \in [n]} A_{ij}(t) x_j(t) \), where \( A_{ij}(t) \in [0,1] \) represents the proportion of individual \( j \)'s garbage that is dumped onto individual \( i \) at time \( t \). This ensures that \( \sum_{i \in [n]} A_{ij}(t) = 1 \). A vector is \emph{stochastic} if all entries are nonnegative and add up to 1. A square matrix is \emph{row-stochastic} if each row is stochastic, and \emph{column-stochastic} if each column is stochastic. Writing the update mechanism in matrix form:
\begin{equation}\label{model:garbage disposal game}
    x(t+1) = A(t)x(t)
\end{equation}
where
\begin{align*}
    x(t) &= \text{transpose of } (x_1(t), \ldots, x_n(t)) = (x_1(t), \ldots, x_n(t))' \in \mathbb{R}_{\geq 0}^n, \\
    A(t) &\in \mathbb{R}^{n \times n} \text{ is column-stochastic with the } (i,j)\text{-th entry } A_{ij}(t).
\end{align*}
The utility of individual \( i \) at time \( t \) is \( u_i(x_i(t)) \), where \( u_i \) is a decreasing function. This indicates that the more garbage an individual processes, the less utility they derive.

Unlike certain opinion models, such as the voter model, the threshold voter model and the asynchronous Hegselmann-Krause model, where an agent solely updates their opinion at each time step, an agent in the garbage disposal game cannot update their garbage independently if they dump onto others~\cite{li2024imitation,castellano2009statistical,mHK2,lanchier2022consensus,cox1991nonlinear,andjel1992clustering,durrett1992multicolor,durrett1993fixation,liggett1994coexistence,clifford1973model,holley1975ergodic}. In the Hegselmann-Krause (HK) model, an agent updates their opinion by averaging the opinions of their opinion neighbors. In the synchronous HK model, all agents update their opinions at each time step, whereas in the asynchronous HK model, only one agent, uniformly selected at random, updates their opinion at each time step~\cite{mHK,mHK2}. The HK model belongs to averaging dynamics but is not necessarily a garbage disposal game. Moreover, the matrix \( A(t) \) is not always row-stochastic, so the garbage disposal game does not necessarily belong to averaging dynamics.

In this paper, we consider the garbage disposal game, where an individual can dump garbage onto others if and only if they are social neighbors and their garbage differs by at most a confidence threshold $\epsilon > 0.$ In detail, let \( G = ([n], E) \) be an undirected simple social graph with vertex set \( [n] \) and edge set \( E \). An edge \( (i,j) \in E \) symbolizes that agents \( i \) and \( j \) are social neighbors. Let \( G_t = ([n], E_t) \) be a subgraph of \( G \) at time \( t \) with vertex set \( [n] \) and edge set \( E_t = \{(i,j) \in E : |x_i(t) - x_j(t)| \leq \epsilon\} \), recording agents who are social neighbors and whose garbage differs by at most the threshold \( \epsilon \). \( N_i(t) = \{ j \in [n] : (i, j) \in E_t \} \) includes all social neighbors of agent \( i \) at time $t$ whose garbage differs by at most the threshold \( \epsilon \). The proportion of agent \( j \)'s garbage dumping onto agent \( i \) equals \( \frac{1}{|E_t|} \) if \( (i, j) \in E_t \), \( 1 - \frac{|N_i(t)|}{|E_t|}\mathbbm{1}\{E_t\neq\emptyset\} \) if \( i = j \), and 0 otherwise. Namely,
$$
A_{ij}(t) = \frac{1}{|E_t|} \mathbbm{1}\{(i,j) \in E_t\} \quad \text{if } i \neq j, \quad A_{ii}(t) = 1 - \frac{|N_i(t)|}{|E_t|} \mathbbm{1}\{E_t \neq \emptyset\}.
$$
Thus, agents \( i \) and \( j \) dump the same proportion of their garbage onto each other if they are social neighbors and their garbage differs by at most the threshold \( \epsilon \). For instance, if $G_t$ is the star graph of order 6, as shown in Figure~\ref{fig:S6}, then agent 1 can dump \(\frac{1}{5}\) of their garbage onto each of the other agents, and each of the other agents can dump \(\frac{1}{5}\) of their garbage onto agent 1. Note that agent 1 is the center of the star graph and is the only agent emptying their original garbage. The update mechanism is as follows for all \( i \in [n] \):
\begin{equation}\label{model G}
    x_i(t+1) = \frac{\mathbbm{1}\{E_t \neq \emptyset\}}{|E_t|}  \sum_{j \in N_i(t)} x_j(t) + \left(1 - \frac{|N_i(t)|}{|E_t|} \mathbbm{1}\{E_t \neq \emptyset\}\right) x_i(t).
\end{equation}
Observe that the garbage disposal game described in~\eqref{model G} falls within the category of averaging dynamics. We assume that \( x_i(0) \), \( i \in [n] \), are nonnegative real-valued random variables. Let \( a \wedge b \) and \( a \vee b \) denote the \emph{minimum} and \emph{maximum} of \( a \) and \( b \), respectively. A graph \( G \) is termed \(\delta\)-\emph{trivial} if the distance between any two vertices is at most \(\delta\). Let \( V(H) \) represent the vertex set of the graph \( H \). Let \( \mathbbm{1} \) denote the vector with all entries equal to 1. The \emph{convex hull} generated by \( v_1, v_2, \ldots, v_n \in \mathbb{R}^d \) is the smallest convex set containing \( v_1, v_2, \ldots, v_n \). It is defined as follows:
$$ C(\{v_1, v_2, \ldots, v_n\}) = \left\{ v : v = \sum_{i=1}^n \lambda_i v_i \text{ where } (\lambda_i)_{i=1}^n \text{ is a stochastic vector} \right\}. $$
\begin{figure}[h]
    \centering
    \includegraphics[width=0.4\textwidth]{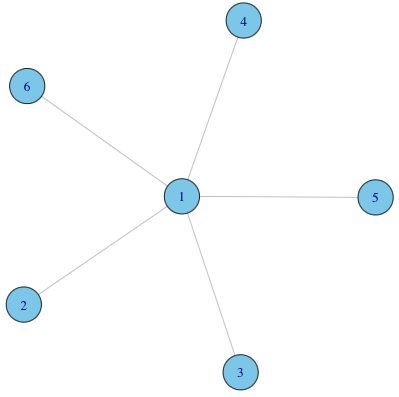}
    \caption{Star graph of order 6}
    \label{fig:S6}
\end{figure}

\section{Main results}

 The $\epsilon$-triviality of the graph $G_s$ implies $G = G_s$, which is equivalent to the case without the threshold $\epsilon$. We will later prove that the graph $G_t$ preserves $\delta$-triviality over time for all $\delta > 0$, leading to $G_t = G$ for all $t \geq s$. Thus, from time $s$ onward, the garbage disposal game is played on the social graph $G$. Consequently, Theorem~\ref{Thm:average} shows that, without the threshold $\epsilon$, the garbage amounts of all agents will converge to the initial average garbage on a connected social graph that is not a star. Note that setting $\epsilon = \infty$ is equivalent to having no threshold \(\epsilon\).

\begin{theorem}\label{Thm:average}
    Assume that 
    \begin{itemize}
        \item the social graph \( G \) is connected and not a star graph, and
        \item the graph \( G_s \) is \(\epsilon\)-trivial for some time \( s \geq 0 \).
    \end{itemize}
    Then, \( x_i(t) \) converges to \( \frac{1}{n} \sum_{k \in [n]} x_k(0) \) as \( t \to \infty \) for all \( i \in [n] \).
\end{theorem}

\section{Properties of the garbage disposal game under~\autoref{model G}}
Since \( C((x_i(t))_{i \in [n]}) \supset C((x_i(t+1))_{i \in [n]}) \), the graph \( G_t \) preserves \(\delta\)-triviality over time.

\begin{lemma}\label{lemma:delta-trivial preserving}
    If \( G_t \) is \(\delta\)-trivial, then \( G_{t+1} \) is also \(\delta\)-trivial for all \(\delta > 0\).
\end{lemma}

We find a nonincreasing function that helps substantiate the asymptotic stability of \(x_i(t)\) as \(t \to \infty\).

\begin{lemma}\label{lemma:supermartingale}
    Let 
    $$Z_t = \sum_{i,j \in [n]} \left[\epsilon^2 \wedge (x_i(t) - x_j(t))^2\right] \vee \epsilon^2 \mathbbm{1}\{(i,j) \notin E\}.$$ 
    Then,
    $$Z_t - Z_{t+1} \geq 4 \sum_{i \in [n]} (|E_t| - |N_i(t)|)\left(x_i(t) - x_i(t+1)\right)^2.$$
\end{lemma}

\begin{proof}
    It is clear that $Z_t-Z_{t+1}\geq 0$ if $E_t=\emptyset.$ Assuming that $E_t\neq\emptyset,$ let $x_i = x_i(t)$, $x_i^\star = x_i(t+1)$ and $N_i = N_i(t)$ for all $i \in [n]$. Then,
    \begin{align}\label{Eq:Delta Z_t}
        Z_t - Z_{t+1} &\geq \sum_{i \in [n]} \sum_{j \in N_i} \left[(x_i - x_j)^2 - (x_i^\star - x_j^\star)^2\right]
    \end{align}
    where
    \begin{align*}
        (x_i - x_j)^2 - (x_i^\star - x_j^\star)^2 &= (x_i - x_i^\star + x_i^\star - x_j)^2 - (x_i^\star - x_j + x_j - x_j^\star)^2 \\
        &\hspace{0.5cm}= (x_i - x_i^\star)^2 - (x_j - x_j^\star)^2 + 2(x_i - x_i^\star)(x_i^\star - x_j) \\
        &\hspace{0.5cm}\quad - 2(x_i^\star - x_j)(x_j - x_j^\star).
    \end{align*}
    Observe that
    \begin{align*}
        &\sum_{i \in [n]} \sum_{j \in N_i} \left[(x_i - x_i^\star)^2 - (x_j - x_j^\star)^2\right] = \sum_{i \in [n]} |N_i|(x_i - x_i^\star)^2 - \sum_{j \in [n]} |N_j|(x_j - x_j^\star)^2 = 0,\\
        &\sum_{j \in N_i} (x_i^\star - x_j) = |N_i|x_i^\star - \frac{|E_t|}{|E_t|} \sum_{j \in N_i}  x_j =|N_i|x_i^\star - |E_t|\left[x_i^\star - \left(1 - \frac{|N_i|}{|E_t|}\right) x_i\right] \\
        &\hspace{2.1cm}= |N_i|(x_i^\star - x_i) - |E_t|(x_i^\star - x_i) = (|E_t| - |N_i|)(x_i - x_i^\star),\\
        &(x_i^\star - x_j)(x_j - x_j^\star) = (x_i^\star - x_i)(x_j - x_j^\star) + (x_i - x_j)(x_j - x_j^\star),\\
        &\sum_{i \in N_j}(x_i - x_j) = \frac{|E_t|}{|E_t|} \sum_{i \in N_j}  x_i - |N_j| x_j\\
        &\hspace{2.1cm}= |E_t|\left[x_j^\star - \left(1 - \frac{|N_j|}{|E_t|}\right) x_j\right] - |N_j| x_j = -|E_t|(x_j - x_j^\star),\\
        &-2(x_i^\star - x_i)(x_j - x_j^\star) = [(x_i^\star - x_i) - (x_j - x_j^\star)]^2 - (x_i^\star - x_i)^2 - (x_j - x_j^\star)^2\\
        &\hspace{3.4cm}\geq -(x_i^\star - x_i)^2 - (x_j - x_j^\star)^2.
    \end{align*}
    So, 
    \begin{align*}
        \eqref{Eq:Delta Z_t} &= 2 \sum_{i \in [n]} (|E_t| - |N_i|)(x_i - x_i^\star)^2 - 2 \sum_{i \in [n]} |N_i| (x_i^\star - x_i)^2 + 2 \sum_{j \in [n]} |E_t|(x_j - x_j^\star)^2 \\
        &= 4 \sum_{i \in [n]} (|E_t| - |N_i|)(x_i - x_i^\star)^2.
    \end{align*}
\end{proof}

It turns out that asymptotic stability holds for \(x_i(t)\) as \(t \to \infty\) as long as \(|E_t| - |N_i(t)| \geq 1\) after some time.

\begin{lemma}
    $x_i(t)$ is asymptotically stable as $t\to\infty$ if $|E_t| - |N_i(t)|\geq 1$ after some time.
\end{lemma}

\begin{proof}
    We claim that $(x_i(t))_{t\geq 0}$ is a Cauchy sequence. By Lemma~\ref{lemma:supermartingale}, $(Z_t)_{t\geq 0}$ is a nonnegative supermartingale. By the martingale convergence theorem, $Z_t$ converges to some random variable $Z_\infty$ with finite expectation as $t\to\infty.$ Assuming without loss of generality that $|E_t| - |N_i(t)|\geq 1$ for all $t\geq 0$, it follows from Lemma~\ref{lemma:supermartingale} that for all $k\geq 0,$
\begin{align*}
    &(x_i(t) - x_i(t+k))^2=\left[\sum_{j=t}^{t+k-1}(x_i(j) - x_i(j+1))\right]^2\leq \left(\sum_{j=t}^{t+k-1}|x_i(j) - x_i(j+1)|\right)^2\\
    &\hspace{1cm}\leq k \sum_{j=t}^{t+k-1}(x_i(j) - x_i(j+1))^2\leq  k \sum_{j=t}^{t+k-1}(|E_t|-|N_i(j)|)(x_i(j) - x_i(j+1))^2\\
    &\hspace{1cm} \leq k \sum_{i\in [n]}\sum_{j=t}^{t+k-1}(|E_t|-|N_i(j)|)(x_i(j) - x_i(j+1))^2\\
    &\hspace{1cm} \leq\frac{k}{4} \sum_{j=t}^{t+k-1}(Z_j - Z_{j+1}) = \frac{k}{4}\left(Z_t - Z_{t+k}\right) \to 0\ \hbox{as}\ t\to\infty.
\end{align*}
Hence, $(x_i(t))_{t\geq 0}$ is a Cauchy sequence. This implies $x_i(t)$ converges to some random variable $x_i$ as $t\to\infty$.

\end{proof}

Since the total garbage is conserved over time, asymptotic stability holds for \(x_i(t)\) as \(t \to \infty\) for all \(i \in [n]\) if at most one \(j \in [n]\) does not satisfy \(|E_t| - |N_j(t)| \geq 1\) after some time.

\begin{lemma}
    \(x_i(t)\) is asymptotically stable as \(t \to \infty\) for all \(i \in [n]\) if there is at most one \(j \in [n]\) that does not satisfy \(|E_t| - |N_j(t)| \geq 1\) after some time.
\end{lemma}

Lemmas~\ref{lemma: CF formula} and~\ref{lemma: CheegerInq} help establish an inequality for a \(\delta\)-nontrivial graph.

\begin{lemma}[Courant-Fischer Formula \cite{horn2012matrix}]
\label{lemma: CF formula}
 Assume that \( Q \) is a symmetric matrix with eigenvalues \( \lambda_1 \leq \lambda_2 \leq \cdots \leq \lambda_n \) and corresponding eigenvectors \( v_1, v_2, \ldots, v_n \). Let \( S_k \) be the vector space generated by \( v_1, v_2, \ldots, v_k \), and \( S_0 = \{0\} \). Then,
\[
\lambda_k = \min \{ x' Q x : \|x\| = 1, \, x \in S_{k - 1}^{\perp} \}.
\]
\end{lemma}

\begin{lemma}[Cheeger's Inequality \cite{beineke2004topics}]
\label{lemma: CheegerInq}
 Assume that \( G = (V, E) \) is an undirected graph with the Laplacian \( \mathscr{L} \). Define
\[
i(G) = \min \left\{ \frac{|\partial S|}{|S|} : S \subset V, \, 0 < |S| \leq \frac{|G|}{2} \right\}
\]
where \( \partial S = \{(u,v) \in E : u \in S, \, v \in S^c\} \). Then,
\[
2i(G) \geq \lambda_2(\mathscr{L}) \geq \frac{i^2(G)}{2 \Delta(G)}
\]
where \( \Delta(G) = \) maximum degree of \( G \).

\end{lemma}

\begin{lemma}\label{lemma:delta-nontrivial}
    If some component $H$ in graph $G_t$ is $\delta$-nontrivial, then $$\sum_{i\in V(H)}(x_i(t)-x_i(t+1))^2> \frac{2\delta^2}{|V(H)|^6|E_t|^2}.$$
    In particular, $$\sum_{i\in [n]}(x_i(t)-x_i(t+1))^2>\frac{2\delta^2}{n^6|E|^2}.$$
\end{lemma}

\begin{proof}
    We assume without loss of generality that graph $G_t$ is connected. Since $\mathbb{R}^n=W\oplus W^\perp$ for~$\mathbbm{1} \in \mathbb{R}^n$ and~$W = \text{Span}(\{\mathbbm{1} \})$, write $x(t)=c\mathbbm{1}+\hat{c}u$ where $c$ and $\hat{c}$ are constants and $u=(u_1,\ldots,u_n)'\in W^\perp$ is a unit vector. We claim that $\hat{c}^2>\delta^2/2.$ Assume by contradiction that $\hat{c}^2\leq\delta^2/2.$  Then for all $i,j\in [n]$, $$(x_i(t)-x_j(t))^2=\hat{c}^2(u_i-u_j)^2\leq 2\hat{c}^2(u_i^2+u_j^2)\leq 2\hat{c}^2\leq\delta^2,$$ contradicting the $\delta$-nontriviality of $G_t.$

    Since $$x(t+1)=(I-\frac{1}{|E_t|}\mathscr{L}_t)x(t),\ x(t)-x(t+1)=\frac{1}{|E_t|}\mathscr{L}_t x(t)=\frac{\hat{c}}{|E_t|}\mathscr{L}_t u.$$ It follows from Lemmas~\ref{lemma: CF formula} and~\ref{lemma: CheegerInq} that $$\norm{x(t)-x(t+1)}^2=\frac{\hat{c}^2}{|E_t|^2}u'\mathscr{L}_t^2u\geq \frac{\hat{c}^2}{|E_t|^2} \lambda_2(\mathscr{L}_t^2)=\frac{\hat{c}^2}{|E_t|^2} \lambda_2^2(\mathscr{L}_t)>\frac{2\delta^2}{n^6|E_t|^2}$$ where $$ \lambda_2 (\mathscr{L}_t) \geq \frac{i^2 (G_t)}{2 \Delta(G_t)} > \frac{(2/n)^2}{2n} = \frac{2}{n^3}. $$
\end{proof}

\begin{proof}[\bf Proof of Theorem~\ref{Thm:average}]
    We assume without loss of generality that \( G_0 \) is \(\epsilon\)-trivial. It follows from Lemma~\ref{lemma:delta-trivial preserving} that the \(\epsilon\)-triviality of the graph \( G_0 \) is preserved over time; therefore, \( G_t = G \) for all \( t \geq 0 \). Thus, if the social graph \( G \) is connected and not a star graph, it implies that \( |E_t| - |N_i(t)| \geq 1 \) for all \( i \in [n] \) and \( t \geq 0 \). 

We claim that \( G_t \) becomes \(\delta\)-trivial after some time for all \(\delta > 0\). Assume by contradiction that \( G_t \) is \(\delta\)-nontrivial for infinitely many time \( t \). From Lemmas~\ref{lemma:supermartingale} and~\ref{lemma:delta-nontrivial}, it follows that 
\[
Z_t - Z_{t+1} \geq \sum_{i \in [n]} (x_i(t) - x_i(t+1))^2 > \frac{2\delta^2}{n^6 |E|^2} \quad \text{for infinitely many time}\ t,
\]
which contradicts the fact that \( Z_t - Z_{t+1} \to 0 \) as \( t \to \infty \). Therefore, all \( x_i(t) \), \( i \in [n] \), approach the same random variable as \( t \to \infty \). Since the total amount of garbage is conserved over time, we have \( x_i(t) \to \frac{1}{n} \sum_{k \in [n]} x_k(0) \) as \( t \to \infty \).

\end{proof}

\section{Statements and Declarations}
\subsection{Competing Interests}
The author is funded by NSTC grant.

\subsection{Data availability}
No associated data was used.

\begin{thebibliography}{10}

\bibitem{andjel1992clustering}
E.~D. Andjel, T.~M. Liggett, and T.~Mountford.
\newblock Clustering in one-dimensional threshold voter models.
\newblock {\em Stochastic processes and their applications}, 42(1):73--90, 1992.

\bibitem{beineke2004topics}
L.~W. Beineke, R.~J. Wilson, P.~J. Cameron, et~al.
\newblock {\em Topics in algebraic graph theory}, volume 102.
\newblock Cambridge University Press, 2004.

\bibitem{castellano2009statistical}
C.~Castellano, S.~Fortunato, and V.~Loreto.
\newblock Statistical physics of social dynamics.
\newblock {\em Reviews of modern physics}, 81(2):591, 2009.

\bibitem{clifford1973model}
P.~Clifford and A.~Sudbury.
\newblock A model for spatial conflict.
\newblock {\em Biometrika}, 60(3):581--588, 1973.

\bibitem{cox1991nonlinear}
J.~Cox and R.~Durrett.
\newblock Nonlinear voter models.
\newblock In {\em Random Walks, Brownian Motion, and Interacting Particle Systems: A Festschrift in Honor of Frank Spitzer}, pages 189--201. Springer, 1991.

\bibitem{durrett1992multicolor}
R.~Durrett.
\newblock Multicolor particle systems with large threshold and range.
\newblock {\em Journal of Theoretical Probability}, 5:127--152, 1992.

\bibitem{durrett1993fixation}
R.~Durrett and J.~E. Steif.
\newblock Fixation results for threshold voter systems.
\newblock {\em The Annals of Probability}, pages 232--247, 1993.

\bibitem{hirai2006coalition}
T.~Hirai, T.~Masuzawa, and M.~Nakayama.
\newblock Coalition-proof nash equilibria and cores in a strategic pure exchange game of bads.
\newblock {\em Mathematical Social Sciences}, 51(2):162--170, 2006.

\bibitem{holley1975ergodic}
R.~A. Holley and T.~M. Liggett.
\newblock Ergodic theorems for weakly interacting infinite systems and the voter model.
\newblock {\em The annals of probability}, pages 643--663, 1975.

\bibitem{horn2012matrix}
R.~A. Horn and C.~R. Johnson.
\newblock {\em Matrix analysis}.
\newblock Cambridge university press, 2012.

\bibitem{lanchier2022consensus}
N.~Lanchier and H.-L. Li.
\newblock Consensus in the {H}egselmann--{K}rause model.
\newblock {\em Journal of Statistical Physics}, 187(3):1--13, 2022.

\bibitem{mHK}
H.-L. Li.
\newblock Mixed {H}egselmann-{K}rause dynamics.
\newblock {\em Discrete and Continuous Dynamical Systems - B}, 27(2):1149--1162, 2022.

\bibitem{mHK2}
H.-L. Li.
\newblock Mixed {H}egselmann-{K}rause dynamics {II}.
\newblock {\em Discrete and Continuous Dynamical Systems - B}, 28(5):2981--2993, 2023.

\bibitem{li2024imitation}
H.-L. Li.
\newblock An imitation model based on the majority.
\newblock {\em Statistics \& Probability Letters}, 206:110007, 2024.

\bibitem{liggett1994coexistence}
T.~M. Liggett.
\newblock Coexistence in threshold voter models.
\newblock {\em The Annals of Probability}, 22(2):764--802, 1994.

\end{thebibliography}

\end{document}